\documentclass[10pt,final]{article}

\usepackage{amssymb,amsmath,amsthm,authblk,graphicx,subfig,url,mathptmx}

\newtheoremstyle{theorem}
  {10pt}		  
  {10pt}  
  {\sl}  
  {\parindent}     
  {\bf}  
  {. }    
  { }    
  {}     
\theoremstyle{theorem}
\newtheorem{theorem}{Theorem}

\newtheorem{lemma}{Lemma}

\newtheoremstyle{defi}
  {10pt}		  
  {10pt}  
  {\rm}  
  {\parindent}     
  {\bf}  
  {. }    
  { }    
  {}     
\theoremstyle{defi}
\newtheorem{definition}{Definition}
\newtheorem{example}{Example}

\begin{document}

\author[1]{\textbf{B{\"u}lent Altunkaya}\footnote{Correspondence: bulent.altunkaya@ahievran.edu.tr}}
\author[1]{\textbf{Ferda\u{g} Kahraman Aksoyak}}
\author[2]{\textbf{Levent Kula}}
\author[1]{\textbf{Cahit Aytekin}}
\affil[1]{Department of Mathematics, Faculty of  Education, University of Ahi Evran}
\affil[2]{Department of Mathematics, Faculty of  Arts and Sciences, University of Ahi Evran}

\title{\textbf{On rectifying slant helices in Euclidean 3-space} \footnotetext {This research has been supported by Ahi Evran University: PYO-EGF.4001.15.001}}
\date{}
\maketitle
\begin{abstract}
In this paper, we study the position vector of rectifying slant helices in $E^3$. First, we have found the general equations of the curvature and the torsion of rectifying slant helices. After that, we have constructed a second order linear differential equation and by solving the equation, we have obtained a family of rectifying slant helices which lie on cones.
\end{abstract}

{\bf Key Words:} Rectifying Curve, Curvature, Torsion,  Slant Helix, Cone.

{\bf 2010 AMS Subject Classification :} 53A04, 53A05
\section{Introduction}
\label{sec:intro}

In classical differential geometry; a general helix in the Euclidean 3-space, is a curve which makes a constant angle with a fixed direction.

The notion of rectifying curve has been introduced by Chen \cite{chen1,chen2}. Chen showed, under which conditions, the position vector of a unit speed curve lies in its rectifying plane. He also stated the importance of rectifying curves in Physics.

On the other hand, the notion of slant helix was introduced by Izuyama and Takeuchi \cite{izu,izu2}. They showed, under which conditions, a unit speed curve is a slant helix. Later, Ahmet T. Ali published a paper in which position vectors of some slant helices were shown \cite{ali}. In \cite{kula1,kula2}, L. Kula, et al studied the spherical images under both tangent and binormal indicatrices of slant helices and obtained that the spherical images of a slant helix  are spherical helices. 

The papers mentioned above led us to study on the notion of rectifying slant helices. We began with finding the equations of curvature and torsion of a rectifying slant helix. After that, we constructed a second order linear differential equation to determine position vector of a rectifying slant helix. By solving this equation for some special cases, we obtained a unit speed family of rectifying slant helices which lie on cones. 

\section{Preliminaries}
\label{sec:intro}

The Euclidean 3-space $E^3$ is the real vector space $R^3$ with the metric
\begin{equation*}
g=d{ x }_{ 1 }^{ 2 }+d{ x }_{ 2 }^{ 2 }+d{ x }_{ 3 }^{ 2 },
\end{equation*}
where $(x_1,x_2,x_3)$ is a rectangular coordinate system of $E^3$.

A curve $\alpha :I\subset R\longrightarrow { E}^{ 3 }$ is said to be parametrized by the arclength parameter s, if $g\left( { \alpha  }^{ ' }\left( s \right) ,{ \alpha  }^{ ' }\left( s \right)  \right) =1$, where ${ \alpha  }^{ \prime  }\left( s \right)=d\alpha/ds$. Then, we  call $\alpha$ unit speed.  Consider unit-speed space curve $\alpha$ has at least four continuous derivatives, then $\alpha$ has a natural frame called Frenet Frame  with the equations below,

\begin{equation*}
\begin{matrix} \begin{matrix} { t }^{ ' }=\kappa n \\ { n }^{ ' }=-\kappa t+\tau b \\ { b }^{ ' }=-\tau n, \end{matrix} \end{matrix}
\end{equation*}
where $\kappa$ is the curvature, $\tau$ is the torsion, and $\left\{ t,n,b \right\}$ is the Frenet Frame of the curve $\alpha$. We denote unit tangent vector field with $t$, unit principal normal vector field with n, and the unit binormal vector field with $b$. It is possible in general, that ${ t }^{ ' }(s)=0$ for some $s\in I$; however, we assume that this never happens.

\begin{definition}
A curve is called a slant helix if its principal normal vector field makes a constant angle with a fixed line in space.
\end{definition}

\begin{theorem}
A unit speed curve $\alpha$ is a slant helix if and only if the geodesic curvature of the spherical image of the principal normal indicatrix of $\alpha$ which is
\begin{equation*}
\sigma(s)=\left( \frac { { \kappa  }^{ 2 } }{ { { \left( { \kappa  }^{ 2 }+{ \tau  }^{ 2 } \right)  } }^{ { 3 }/{ 2 } } } { \left( \frac { \tau  }{ \kappa  }  \right)  }^{ ' } \right) \left( s \right)
\end{equation*}
is constant \cite{izu,izu2}.
\end{theorem}

 A unit speed curve $\alpha$ is called rectifying curve when the position vector of it always lie in its rectifying plane. So, for a rectifying curve we can write

\begin{equation*}
\alpha \left( s \right) =\lambda \left( s \right) t\left( s \right) +\mu \left( s \right) b\left( s \right).
\end{equation*}

\begin{theorem}
A unit speed curve $\alpha$ is congruent to a rectifying curve if and only if
\begin{equation*}
\frac { \tau(s)  }{ \kappa (s) } ={ c }_{ 1 }s+{ c }_{ 2 }
\end{equation*}
for some constants $c_1$  and $c_2$, with $c_1\neq0$ \cite{chen1,chen2}.
\end{theorem}

\section{Rectifying Slant Helices in $E^3$}
\label{sec:intro}

If the position vector of a unit speed slant helix always lies in its rectifying plane we  call it a rectifying slant helix. For a rectifying slant helix we  have the following theorem.

\begin{theorem}
Let $\alpha$ be a unit speed curve in $E^3$. Then, $\alpha(s)$ is a rectifying slant helix if and only if the curvature and torsion of the curve satisfies the equations below;
\begin{equation}
\kappa (s)=\frac { { c }_{ 3 }  }{ { { \left( 1+{ \left( { c }_{ 1 }s+{ c }_{ 2 } \right)  }^{ 2 } \right)  } }^{ { 3 }/{ 2 } } } ,\tau (s)=\frac { { c }_{ 3 }  \left( { c }_{ 1 }s+{ c }_{ 2 } \right)  }{ { { \left( 1+{ \left( { c }_{ 1 }s+{ c }_{ 2 } \right)  }^{ 2 } \right)  } }^{ { 3 }/{ 2 } } },
\end{equation}
where $c_1 \neq 0,c_2\in R$, $\theta \neq 0+k \pi/2, k\in Z$, and $c_3 \in R^+.$
\end{theorem}

\begin{proof}
Let $\alpha$ be a unit speed rectifying slant helix in $E^3$, then the equations in Theorem 1, and Theorem 2 exists. If we combine them then we  have
\begin{equation*}
m =\frac { { c }_{ 1 } }{ { { \kappa \left( 1+{ \left( { c }_{ 1 }s+{ c }_{ 2 } \right)  }^{ 2 } \right)  } }^{ { 3 }/{ 2 } } }.\\
\end{equation*}
where $m$ is a constant. So we can write $\kappa$ as follows

\begin{equation*}
\kappa(s) =\frac { { c }_{ 3 }   }  { { {\left( 1+{ \left( { c }_{ 1 }s+{ c }_{ 2 } \right)  }^{ 2 } \right)  } } ^{ { 3 }/{ 2 } } },
\end{equation*}
then, from Theorem 2

\begin{equation*}
\tau (s)=\frac { { c }_{ 3 } \left( { c }_{ 1 }s+{ c }_{ 2 } \right)  }{ { { \left( 1+{ \left( { c }_{ 1 }s+{ c }_{ 2 } \right)  }^{ 2 } \right)  } }^{ { 3 }/{ 2 } } }.
\end{equation*}
where $c_3=\left|  c_1/m \right|$.

Conversely, it can be easily seen that,  the curvature functions as mentioned above satisfy the equations at Theorem 1 and Theorem 2. So, $\alpha$ is a rectifying slant helix.

\end{proof}

Now, we  give another Theorem by using the definitions of slant helix and rectifying curve to determine $c_3$.
\begin{theorem}
	Let $\alpha$ be a unit speed rectifying slant helix whose principal  normal vector field makes a constant angle with a unit vector $u$, then the curvature and torsion of $\alpha$ satisfy the equations below;
	\begin{equation*}
	\kappa (s)=\frac {\left| c_{ 1 } \tan(\theta)   \right| }{ { { \left( { \left( { c }_{ 1 }s+{ c }_{ 2 } \right)  }^{ 2 }+1 \right)  } }^{ { 3 }/{ 2 } } } ,\quad \tau (s)=\frac { \left| c_{ 1 } \tan(\theta)   \right|  \left( { c }_{ 1 }s+{ c }_{ 2 } \right)  }{ { { \left( { \left( { c }_{ 1 }s+{ c }_{ 2 } \right)  }^{ 2 }+1 \right)  } }^{ { 3 }/{ 2 } } } 
	\end{equation*}
	where $c_1\neq 0,c_2\in R.$
\end{theorem}

\begin{proof}
Let $\alpha$ be a unit speed rectifying slant helix in $E^3$. Then, from the definition of slant helix there is a unit fixed vector $u$  with

	\begin{equation*}
	g(n,u)=\cos(\theta),
	\end{equation*}	
	where $\theta \in R^+$. If we differentiate this equation with respect to $s$, we  have,
	
	\begin{equation*}
	g(-\kappa t+ \tau b,u)  =0.
	\end{equation*}
	If we divide both parts of the equation with $\kappa$, we get
	
	\begin{equation}
	g(-t+ (c_1 s+c_2) b,u)  =0,
	\end{equation}	
	then,
	
	\begin{equation*}
	g(t ,u)=(c_1 s+c_2)g(b,u).
	\end{equation*}
	While $\{t,n,b\}$ is a orthonormal frame we can write,
	
	\begin{equation*}
	v= \lambda_1 t+ \lambda_2 n+ \lambda_3 b,
	\end{equation*}
	with $\lambda_1^2+\lambda_2^2+\lambda_3^2=+1.$ If we make the neccessary calculations we  have,
	\begin{equation*}
	\lambda_1=\mp \frac{ (c_1 s+c_2) \sin  (\theta)  }{ \sqrt{(c_1 s+c_2)^2+1} },\quad    \lambda_2= \cos (\theta) ,\quad  \lambda_3=\pm \frac{ \sin  (\theta)  }{ \sqrt{(c_1 s+c_2)^2+1} }.
	\end{equation*}
	By differentiating (2) we have,
	
	\begin{equation*}
	\pm \frac{ c_1 \sin  (\theta)  }{ \kappa \sqrt{(c_1 s+c_2)^2+1} }-(1+(c_1 s+c_2)^2)\cos(\theta) =0.
	\end{equation*}
	Therefore,
	\begin{equation*}
	\kappa (s)=\frac { \left| c_{ 1 } \tan(\theta)   \right| }{ { { \left( { \left( { c }_{ 1 }s+{ c }_{ 2 } \right)  }^{ 2 }+1 \right)  } }^{ { 3 }/{ 2 } } },
	\end{equation*}
	and
	\begin{equation*}
	\tau (s)=\frac { \left| c_{ 1 } \tan(\theta)   \right|  \left( { c }_{ 1 }s+{ c }_{ 2 } \right)  }{ { { \left( { \left( { c }_{ 1 }s+{ c }_{ 2 } \right)  }^{ 2 }+1 \right)  } }^{ { 3 }/{ 2 } } } .
	\end{equation*}

\end{proof}

\begin{theorem}
     Let $\alpha(s)$ be a unit speed rectifying slant helix. Then, the vector $v$ satisfies the linear vector differential equation of second order as follows;

\begin{equation*}
    v''(s)+\frac { ({ c }_{ 1 } \tan { (\theta)   })^2  }{ { {\left( 1+{ \left( { c }_{ 1 }s+{ c }_{ 2 } \right)  }^{ 2 } \right)  } } ^2 }v(s)=0,
\end{equation*}
where $v=\frac{n'}{\kappa}$.

\begin{proof}
	Let $\alpha$ be a unit speed rectifying slant helix then we can write frenet equations as follows,
	\begin{equation}
	\begin{matrix} \begin{matrix} { t }^{ ' }=\kappa n \\ { n }^{ ' }=-\kappa t+f \kappa b \\ { b }^{ ' }=-f \kappa n, \end{matrix} \end{matrix}
	\end{equation}
	where $f(s)=c_1 s+c_2$. If we divide second equation by $\kappa$ we  have,
		\begin{equation}
	    \frac{n'}{\kappa}=- t+f b.
		\end{equation}
	By differentiating (4), we have
	\begin{equation}
	c_1 b=\left( \frac{n'}{\kappa}\right) '+\kappa (1+f^2)n.
	\end{equation}
	By differentiating (5) and using (3) we  have
	   	\begin{equation}
		\left( \frac{n'}{\kappa}\right) ''+\kappa (1+f^2)n'+\left[ \left( \kappa (1+f^2)\right )'+c_1 f \kappa \right] n=0,
		\end{equation}
		with the necessary calculations we  easily see
		\begin{equation*}
		\left( \kappa (1+f^2)\right )'+c_1 f \kappa =0.
		\end{equation*}
		So we  have (6) as follows,
			\begin{equation}
			\left( \frac{n'}{\kappa}\right) ''+\kappa (1+f^2)n'=0.
			\end{equation}
			Let us denote $\frac{n'}{\kappa}=v$. Then (7) becomes to
			\begin{equation}
			 v''+\frac { ({ c }_{ 1 } \tan { (\theta)   })^2  }{ { {\left( 1+{ \left( { c }_{ 1 }s+{ c }_{ 2 } \right)  }^{ 2 } \right)  } } ^2 }v=0,
			\end{equation}
			this completes the proof.
\end{proof}
\end{theorem}

As we know every component of vector $v=(v_1,v_2,v_3)$ must satisfy (8). We can show

\begin{equation*}
\begin{matrix} \begin{matrix}
&&v_1(s)=-\sqrt{\Bigl( 1+{ f }^{ 2 }\left( s\right)\Bigl)}\sin\left[\sec (\theta)  \arctan\left[f(s)\right]  \right],\\
&&v_2(s)=\sqrt{\Bigl( 1+{ f }^{ 2 }\left( s\right)\Bigl)}\cos\left[\sec (\theta)  \arctan\left[f(s)\right]  \right], \\
&&v_3(s)=0.
\end{matrix} \end{matrix}
\end{equation*}
We can show $v$ is a solution for (8). Therefore, we can write  $n=(n_1,n_2,n_3)$ as follows,

\begin{equation}
\begin{matrix} \begin{matrix}
&&n_1(s)=\int\kappa(s) v_1(s)ds=A_1 \left| c_{ 1 } \right|\sin(\theta) \cos\left[\sec (\theta)  \arctan\left[f(s)\right]  \right],\\
&&n_2(s)=\int\kappa(s) v_2(s)ds=A_2 \left| c_{ 1 } \right|\sin (\theta) \sin\left[\sec (\theta) \arctan\left[f(s)\right]  \right],\\
&&n_3(s)=\cos (\theta).
\end{matrix} \end{matrix}
\end{equation}

On the other hand, Let $\alpha$ be a unit speed rectifying slant helix, whose principal normal vector field makes a constant angle $\theta$ with $e_3$. Then, for its principal normal we can write
\begin{equation*}
<n,e_3>=\cos (\theta).
\end{equation*}
While $n=(n_1,n_2,n_3)$ is a unit vector, $n_1^2+n_2^2+n_3^2=1$. So, $n_1^2+n_2^2=1-\cos^2(\theta)=\sin^2 (\theta)$. Therefore $n$ can be in the form,

\begin{equation}
\begin{matrix}
\begin{matrix}
 &&n_1(s)=\sin (\theta) \cos(h(s))\\
 &&n_2(s)=\sin (\theta) \sin(h(s))\\
 &&n_3(s)=\cos(\theta),
\end{matrix}
\end{matrix}
\end{equation}
where $h(s)$ is a differentiable function.

If we take, $A_1=1/\left| c_{ 1 } \right|, A_2=1/\left| c_{ 1 } \right|, h(s)=\sec (\theta) \arctan\left[f(s)\right]$ at (9), (9) and (10) coincides. Thus, a unit speed rectifying slant helix $\alpha$ can be in the form;

\begin{equation*}
\begin{matrix}
\begin{matrix}
 &&\alpha_1(s)=\sin (\theta) \int \left(\int \kappa(s) \cos\left[\sec (\theta) \arctan\left(c_ 1 s+c_2\right)  \right] ds \right)ds,\\
 &&\alpha_2(s)=\sin (\theta)\int \left(\int \kappa(s) \sin\left[\sec (\theta) \arctan\left(c_ 1 s+c_2\right)  \right]ds \right)ds,\\
 &&\alpha_3(s)=\int \left(\int \kappa(s)\cos(\theta) ds \right) ds,
\end{matrix}
\end{matrix}
\end{equation*}
where $\alpha=(\alpha_1,\alpha_2,\alpha_3)$.

Therefore, we  find  $\alpha$ as follows.

\begin{equation*}
\begin{matrix}
\begin{matrix}
\alpha _{ 1 }(s)=-\frac { \cos {(\theta)  }  }{ c_{ 1 } } \sqrt { 1+\left( c_{ 1 }s+c_{ 2 } \right) ^{ 2 } } \cos { \left[ \sec { (\theta)  } \arctan { \left( c_{ 1 }s+c_{ 2 } \right)  }  \right]  }, \\
\alpha _{ 2 }(s)=-\frac { \cos { (\theta)   }  }{ c_{ 1 } } \sqrt { 1+\left( c_{ 1 }s+c_{ 2 } \right) ^{ 2 } } \sin { \left[ \sec { (\theta ) } \arctan { \left( c_{ 1 }s+c_{ 2 } \right)  }  \right]  }, \\
\alpha _{ 3 }(s)=\frac { 1 }{ c_{ 1 } } \sqrt { 1+\left( c_{ 1 }s+c_{ 2 } \right) ^{ 2 } } \sin { (\theta ). }
\end{matrix}
\end{matrix}
\end{equation*}

Now, we can write a new lemma;

\begin{lemma}
	Let $ \alpha(s) : I \longrightarrow  R^3$ be a space curve with the equation below,
	\begin{equation}
    \begin{split}
    \alpha(s)=-\frac { \sqrt { 1+\left( c_{ 1 }s+c_{ 2 } \right) ^{ 2 }}}{ c_{ 1 } }(&\cos(\theta) \cos \left[\sec { (\theta)  } \arctan { \left( c_{ 1 }s+c_{ 2 } \right)  }  \right],\\
    &\cos(\theta) \sin \left[\sec { (\theta)  } \arctan { \left( c_{ 1 }s+c_{ 2 } \right)  }  \right],\\
    &-\sin(\theta )),
    \end{split}
    \end{equation}
where $\theta\neq \frac { \pi  }{ 2 } +k\pi, k\in Z $, and $c_1\neq0,c_2\in R.$ Then, $\alpha(s)$ is a unit speed rectifying slant helix which lies on the cone
    \begin{equation}
    \tan^2(\theta)\left(x^2+y^2\right)=z^2.
    \end{equation} 
\end{lemma}

\begin{proof}
		With direct calculations we have $g(\alpha',\alpha')=1$, $g(n,n)=1$, and the curvature functions of $\alpha$ as,
		\begin{equation*}
		\begin{matrix}
		\begin{matrix}
		\kappa(s)=\frac{ \left| c_{ 1 } \tan(\theta)   \right|  }{\left((c_1 s+c_2)^2+1\right)^{3/2}},\\
		\tau(s)=\frac{\left| c_{ 1 } \tan(\theta)   \right|  (c_1 s+c_2)}{\left((c_1 s+c_2)^2+1\right)^{3/2}}.
		\end{matrix}
		\end{matrix}
		\end{equation*}	
		with,
		\begin{equation*}
		\frac { { \kappa  }^{ 2 }(s) }{ { { \left( { \kappa  }^{ 2 }(s)+{ \tau  }^{ 2 }(s) \right)  } }^{ { 3 }/{ 2 } } } { \left( \frac { \tau(s)  }{ \kappa(s)  }  \right)  }^{ ' } = \cot(\theta),	
		\end{equation*}
		and
		\begin{equation*}
		\frac{\tau(s)}{\kappa(s)}= c_1 s+c_2.
		\end{equation*}
		So, $\alpha$ is a unit speed spacelike rectifying slant helix. We also have
		\begin{equation*}
		\tan ^{ 2 }{ (\theta ) } \left( { { \alpha  }_{ 1 } }^{ 2 }(s)+{ { \alpha  }_{ 2 } }^{ 2 }(s) \right) -{ { \alpha  }_{ 3 } }^{ 2 }(s)=0,
		\end{equation*}
		then, $\alpha$ lies on the cone above.
\end{proof}

\begin{example}
If we take $c_1=1, c_2=0$, and $\cos(\theta)=1/3$ then, $\tan(\theta)=2\sqrt{2}$. If we put these into (11) and (12), we  have the following equations;
\begin{equation*}
\begin{matrix}
\alpha(s)=\left(-\frac{1}{3} \sqrt{s^2+1} \cos \left(3 \arctan(s)\right),-\frac{1}{3} \sqrt{s^2+1} \sin \left(3 \arctan(s)\right),\frac{2\sqrt{2}}{3} \sqrt{s^2+1}\right),\\
\kappa(s)=\frac{2\sqrt{2} }{\left(s^2+1\right)^{3/2}},
\tau(s)=\frac{2\sqrt{2}s}{\left(s^2+1\right)^{3/2}},\\
8\left(x^2+y^2\right)=z^2.
\end{matrix}
\end{equation*}

\begin{figure}[h!]
	\centering
	\includegraphics[width=0.3\textwidth]{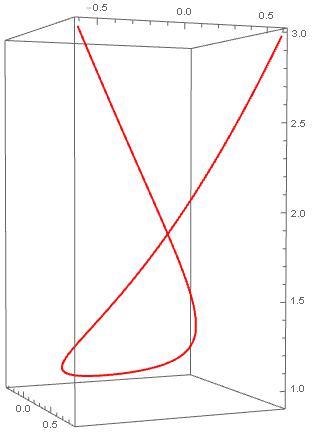}
    \includegraphics[width=0.3\textwidth]{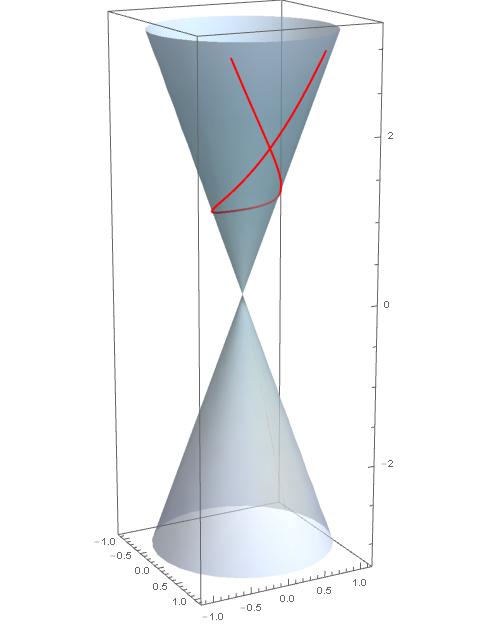}
    \includegraphics[width=0.3\textwidth]{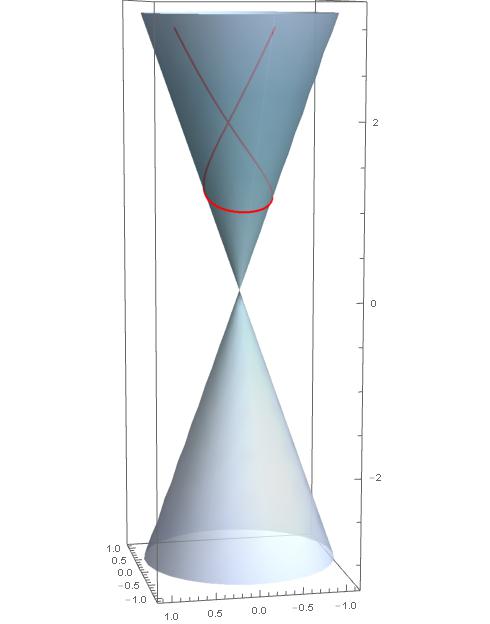}
	\caption{Rectifying Slant Helix on $8\left(x^2+y^2\right)=z^2$}
\end{figure}

\begin{figure}[h!]
	\centering
	\includegraphics[width=0.3\textwidth]{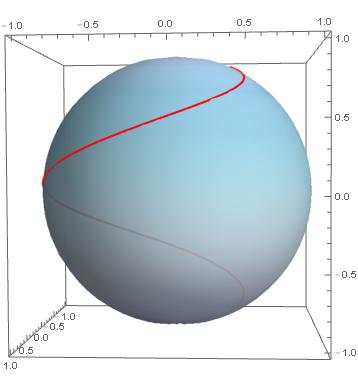}
	\includegraphics[width=0.3\textwidth]{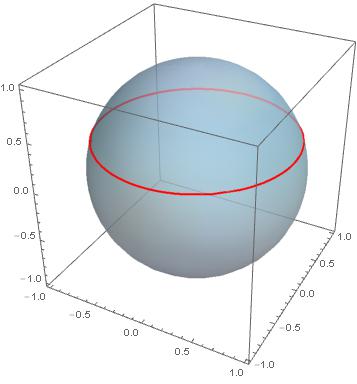}
	\includegraphics[width=0.3\textwidth]{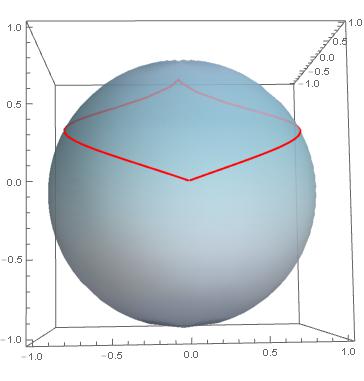}
	\caption{Tangent, Normal,and Binormal indicatrix of $\alpha$ resp.}
\end{figure}

\end{example}       	
\newpage
\begin{example}
	If we take $c_1=1/2, c_2=-1/5$, and $\cos(\theta)=1/10$ then, $\tan(\theta)=\sqrt{99}$. If we put these into (11) and (12), we  have the following equations;
	\begin{equation*}
	\begin{split}
	\beta(s)=\frac{1}{5} \sqrt{\left(\frac{s}{2}-\frac{1}{5}\right)^2+1}\Biggl(&-\cos \left(10 \arctan\left(\frac{s}{2}-\frac{1}{5}\right)\right),\\
	&-\sin \left(10 \arctan\left(\frac{s}{2}-\frac{1}{5}\right)\right),\\
	&\frac{3 \sqrt{11}}{5}\Biggl),
	\end{split}
	\end{equation*}
	\begin{equation*}
	\begin{matrix}
	
	\kappa(s)=\frac{1500 \sqrt{11}}{\Big(5 s (5 s-4)+104\Big)^{3/2}},
	\tau(s)=\frac{150 \sqrt{11} (5 s-2)}{\Big(5 s (5 s-4)+104\Big)^{3/2}},\\
	99\left(x^2+y^2\right)=z^2.
	\end{matrix}
	\end{equation*}

	\begin{figure}[h!]
		\centering
		\includegraphics[height=4.5cm,width=0.3\textwidth]{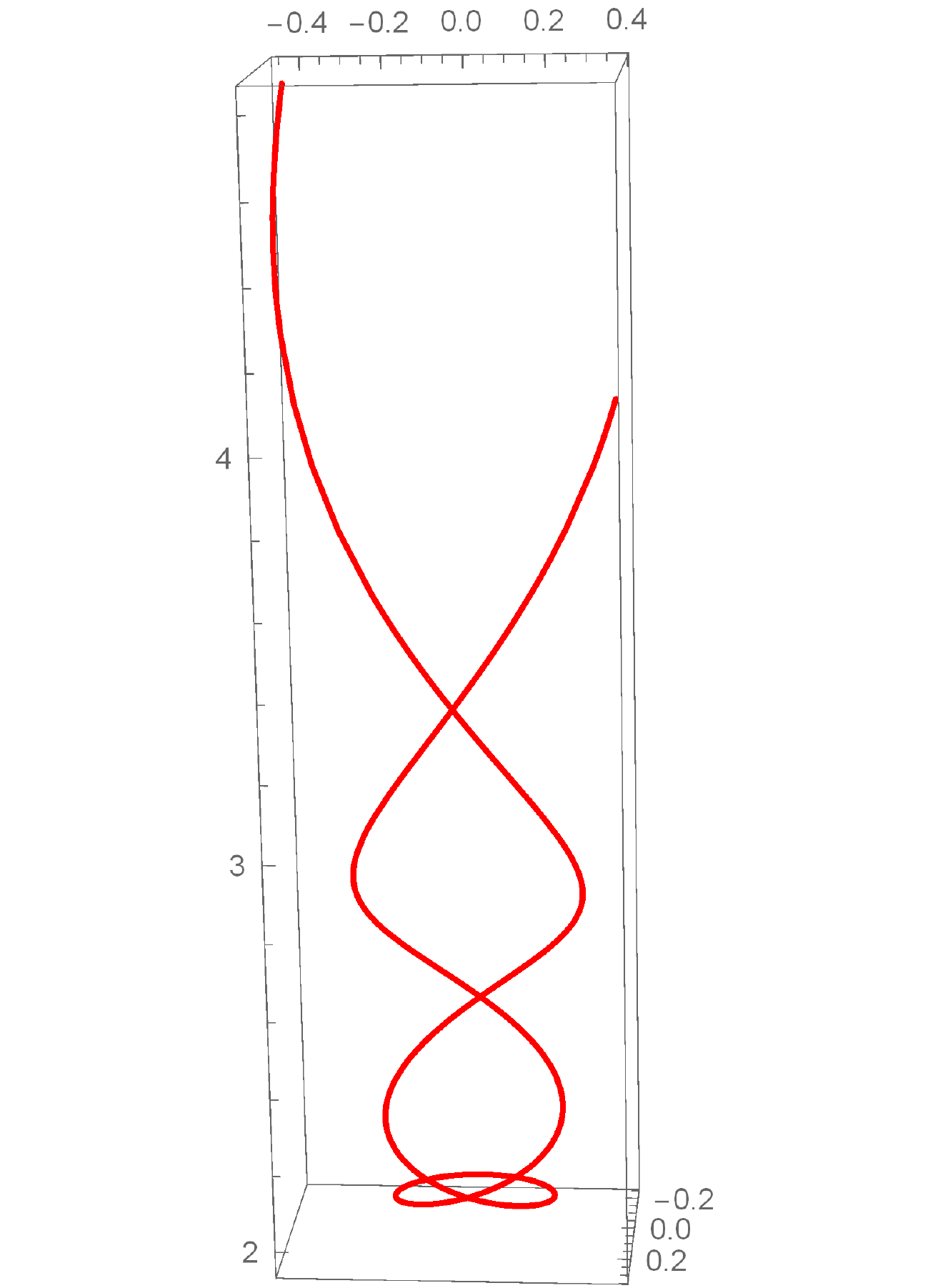}
		\includegraphics[height=4.5cm, width=0.3\textwidth]{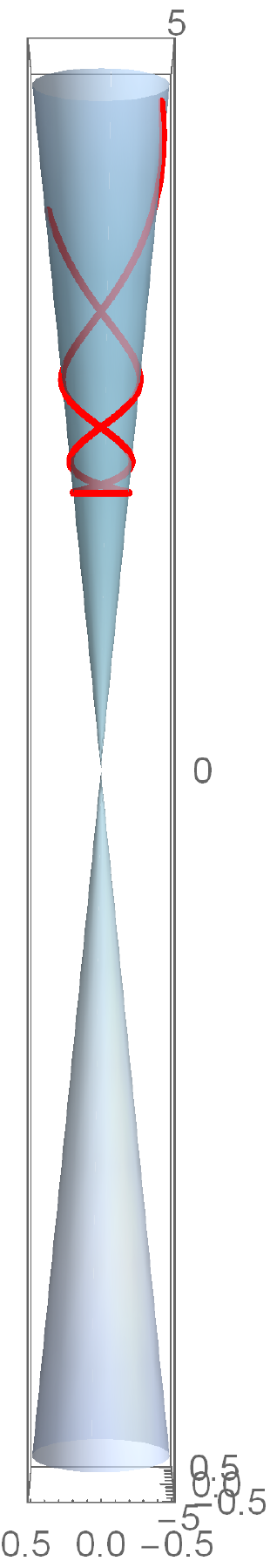}
		\caption{Rectifying Slant Helix on $99\left(x^2+y^2\right)=z^2$}
	\end{figure}
	
	\begin{figure}[h!]
		\centering
		\includegraphics[width=0.3\textwidth]{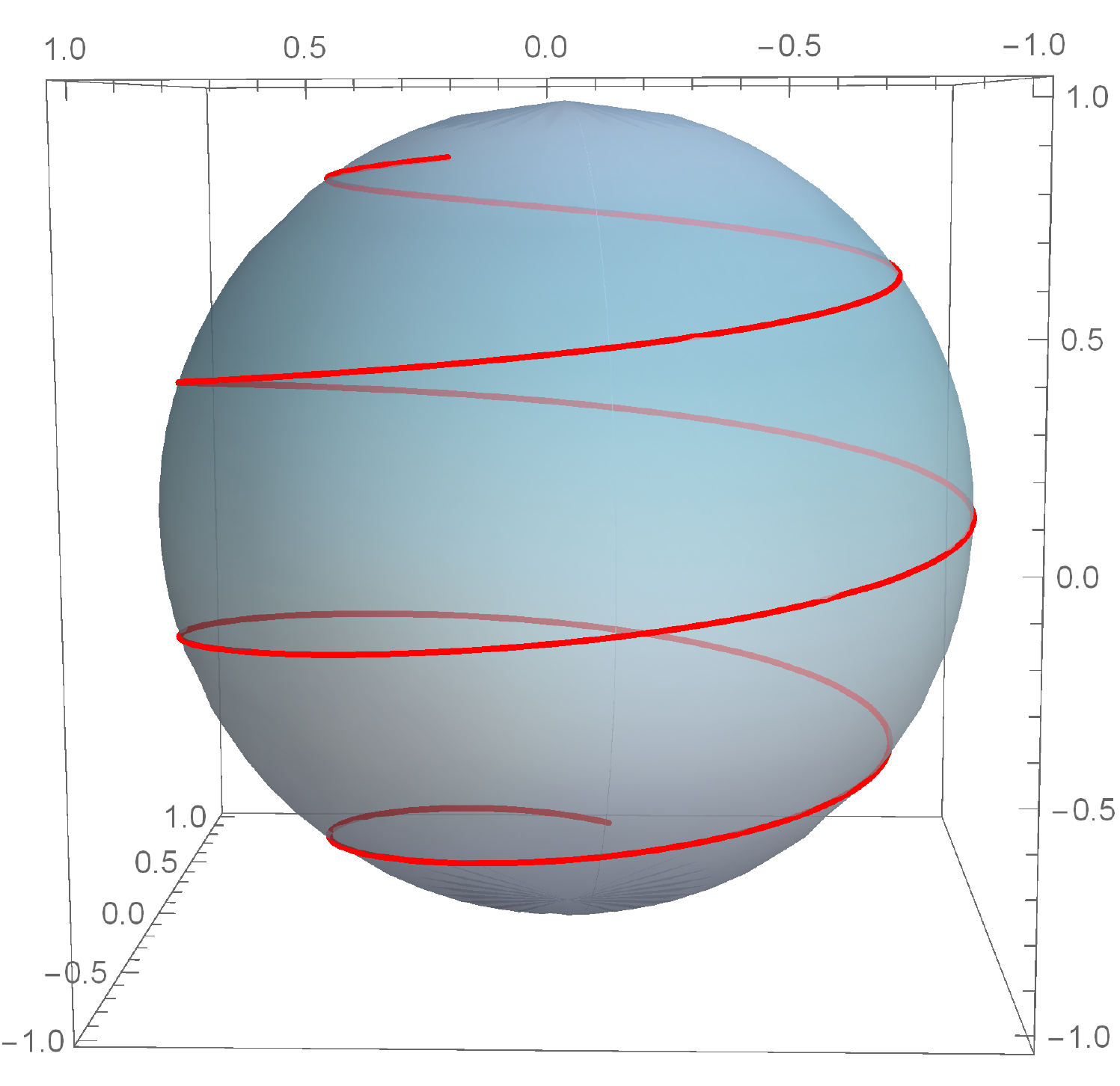}
		\includegraphics[width=0.3\textwidth]{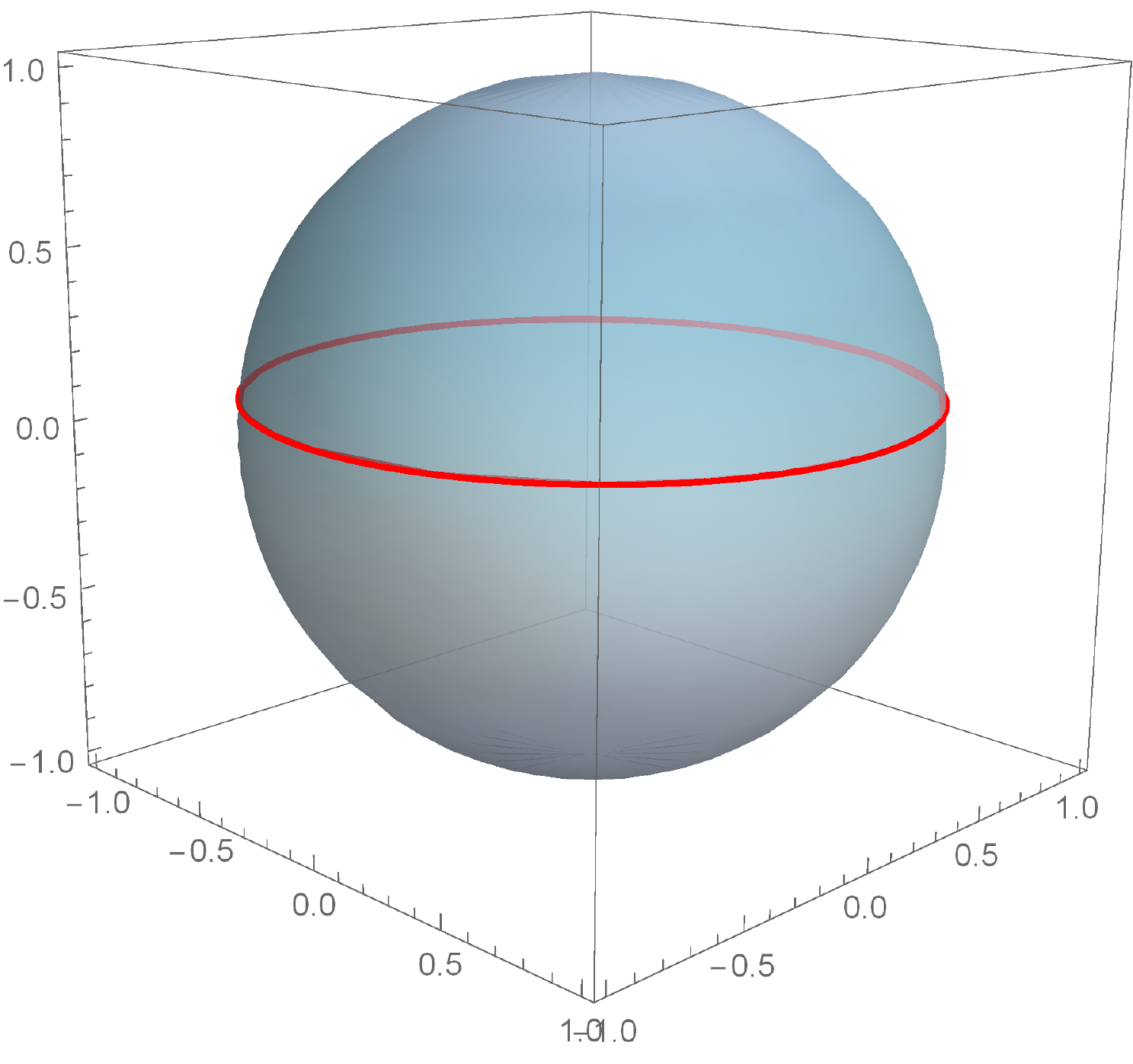}
		\includegraphics[width=0.3\textwidth]{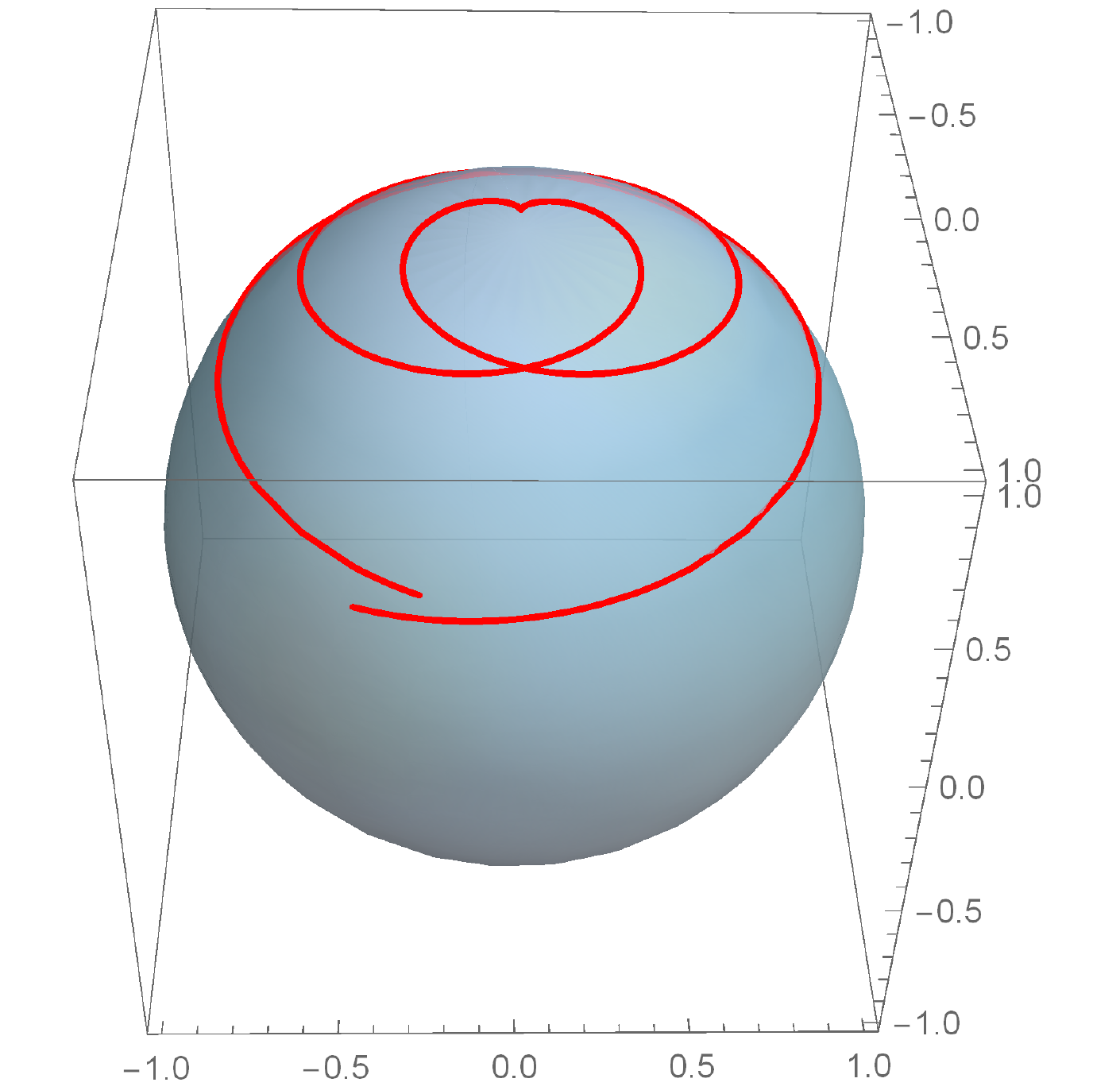}
		\caption{Tangent, Normal,and Binormal indicatrix of $\beta$ resp.}
	\end{figure}
	
\end{example}

\end{document}